\documentclass[12pt]{article}
\usepackage{mathpazo}
\usepackage{amsmath,amsfonts,amssymb}
\usepackage{amsrefs}
\usepackage{amsthm}
\usepackage{latexsym,amsmath,amssymb,amsfonts}
\usepackage{rotating}

\usepackage[margin=3.2cm]{geometry}

\usepackage{tikz}
\usepackage{amsmath}
\usetikzlibrary{calc}
\usetikzlibrary{arrows.meta, decorations.markings}

\usepackage{booktabs}
\usepackage{xypic} \xyoption{all}
\usepackage{amscd}
\usepackage{hyperref} 
\usepackage{euscript}
\usepackage{hhline}
\usepackage{cite}
\usepackage{color}
\usepackage{graphicx,epstopdf}
\usepackage{epsfig}
\usepackage{xcolor}
\usepackage[all,color]{xy}
\newlength{\fighskip} \fighskip=2pt
\newlength{\figvskip} \figvskip=3pt
\usepackage{hyperref}
\numberwithin{equation}{section}
\newcommand{\R}{\mathbb{R}}
\newcommand{\T}{\mathbb{T}}

\newcommand{\Z}{\mathbb{Z}}

\newcommand{\M}{\mathcal{M}}

\newcommand{\A}{\mathcal{A}}
\newcommand{\BA}{\mathbb{A}}

\newcommand{\nn}{\nonumber}

\newcommand{\be}{\begin{eqnarray}}
\newcommand{\ee}{\end{eqnarray}}

\newcommand{\EQN}[1]{\begin{equation*}\begin{split} #1 \end{split}\end{equation*}}

\theoremstyle{plain}
\newtheorem{theorem}{Theorem}[section]
\newtheorem{definition-theorem}{Theorem/Definition}[section]
\newtheorem{lemma}[theorem]{Lemma}
\newtheorem{lemma-definition}[theorem]{Lemma/Definition}
\newtheorem{proposition}[theorem]{Proposition}

\newtheorem{conjecture}[theorem]{Conjecture}

\theoremstyle{definition}
\newtheorem{definition}[theorem]{Definition}

\newtheorem{example}[theorem]{Example}

\theoremstyle{remark}
\newtheorem{remark}[theorem]{Remark}

\allowdisplaybreaks[4]

\begin{document}
\title{Action and rotation number of periodic orbits of area-preserving annulus diffeomorphisms}
\author{Huadi Qu}
\date{}
\maketitle

\abstract{We study periodic orbits for area-preserving surface diffeomorphisms, particularly some global properities related to the action function and rotation numbers. We generalize recent works of Machel Hutchings\cite{Hut16}, proving the existence of periodic orbits with certain action and rotation values.}

\section{Introduction}

\ \ \ We aim to make a preliminary atempt at characterizing the periodic orbits of area-preserving surface diffeomorphisms using two important invariants: the rotation number and the action function. 

Periodic orbits for area-preserving diffeomorphisms on surfaces is a classic subject in dynamical systems, recently converting area-preserving surface diffeomorphisms to Reeb flow via open book decomposition provided a powerful way for understanding periodic orbits of area-preserving surface diffeomorphisms, and the ECH (Embedded Contact Homology) theory proves to be powerful. In particular, the dynamical invariant, action, of periodic orbit plays a very important role. The definition of the action function depends on a preferred closed 1-form, and this dependence is related with another important invariant, the rotation number, of periodic orbits, or more general invariant measure of area-preserving surface diffeomorphisms. It's resonable to put this two important invariant together to describe the distribution of periodic orbits more pricisely.

      The main result of this article is to extend the main therorem in \cite{Hut16} of Machel Hutchings on disk to a theorem on Annulus. Let $D^{2}$ be the unit disk with the standard area form $\omega$. Under the usual polar coordinates $(r,\theta)$, $\omega=2rdrd\theta$. Consider an area-preserving diffeomorphism $f$ on $D^{2}$ that is a pure rotation near $\partial D^{2}$, fix a primitive $\beta$ of $\omega$ such that $\beta|_{\partial D^{2}}=d\theta$. Then there is a smooth function $g$ on $D^2$ such that 
$$f^{*}\beta-\beta=dg$$ 
called the action function of $f$. After fix a boundary condition of $g$, we can define the {\it Calabi invariant} of $f$ as 
$$Cal(f)=\int_{D^2}g\omega.$$

Denote the set of periodic orbit of $f$ as $\mathcal{P}(f)$, for $\gamma=(x_{0},x_{1}=f(x_{0}),...,x_{k}=x_{0})$, define the mean action of $\gamma$ as 
$$ \A(\gamma)=\frac{1}{k}\sum_{i=0}^{k-1}f(x_i).$$


In \cite{Hut16}, Hutchings proved the following theorem:

\begin{theorem}\label{Hut}
Let $\theta_{0}\in\R$, let $f$ be an area-preserving, orientation-preserving diffeomorphism of $D$ which agrees with rotation by angle $\theta_0$ near the boundary. for the fixed $\beta$, fixed the value of the corresopiding action function on the boundary as $g(\partial D^{2})=\theta_{0}$, suppose that $$Cal(f)<\theta_{0},$$ Then 
$ \inf\{\A(\gamma)|\gamma\in\mathcal{P}(f)\}\leq Cal(f).$
\end{theorem}

We remark here that in the origional theorem Hutchings used $\theta\in[0,2\pi],\omega=\frac{1}{2\pi}rdrd\theta$ and let $f(r,\theta)=(r,\theta+2\pi\theta_{0})$ while $r$ near 1. In this article we will assume $\theta\in[0,1]$ and use $\omega=2rdrd\theta$.

The assumption that $f$ is a rigid rotation near the boundary is removed in \cite{Pir24}. And recently \cite{Bra24} extend the theorem to the following Theorem.

\begin{theorem}\label{shangxia}
\be \inf \{\A(\gamma)|\gamma\in\mathcal{P}(f)\}\leq Cal(f)\leq\sup\{\A(\gamma)|\gamma\in\mathcal{P}(f)\}. \ee
\end{theorem}

In \cite{Wei21} Morgan Weiler extended this theorem to the Annulus. Let $\BA=[-1,1]_{p}\times S^1_{q}$ be the annulus with the area form $\omega=dp\wedge dq$, use $\BA_{i}=\{i\}\times S^1, i=\pm1$ to denote the boundary components of $\BA$ respectively. Let $ f $ be an area-preserving, orientation-preserving diffeomorphism on $\BA$ that preserves the boundaries.  We assume that $f$ is rotations near the boundary. That is, there are  $\theta_{-1}, \theta_{1}\in\mathbb{R}$ and a  lift $\tilde{f}$ of $f$ to the universal cover $\tilde{\BA}=[-1,1]\times \R$ of $\BA$,  such that 
\EQN{\tilde{f}(p,q)=\begin{cases}(p, q+\theta_1)& \quad \text{for $p$ sufficiently close to 1,}\\
(p, q+\theta_{-1})& \quad \text{for $p$ sufficiently close to -1.}\end{cases}}

Fix the 1-form $\beta=pdq$, obviously $d\beta =\omega$. Let $g$ be an
action function of $f$ with respect to $\beta$, then $g$ is constant on
both boundaries.  Fix the constant value $g(\BA_{1})=\theta_{1}$, then Weiler state the follwing result:

\begin{conjecture}\label{W}
If the corresponding action function satisfying  
\be
Cal(f) <\max\{g(\BA_1),g(\BA_{-1})\} 
\ee
then the set of periodic orbits of $f$ satisfying 
$$\inf\{\mathcal{A}(\gamma,\beta)|\gamma\in\mathcal{P}(f)\}\leq Cal(f).$$
\end{conjecture}

Notice we state this result as a conjecture because the proof in \cite{Wei21} has some technique gap now. We want to prove a theorem on Annulus from the previous results on Disk. In \cite{Wei21} Appendix part Weiler showed that direct apply theorem \ref{Hut} could only cover a part of the case in which conjecture \ref{W} could apply. We will extend theorem\ref{Hut} and prove some case in conjecture \ref{W}. In order to state our main theorem we need a new variable namely the rotation vector or rotation number of a periodic orbit $\gamma$, denote as $\rho(\gamma)$, the pricise definition will introduce in next section.

\begin{theorem}(main result)

Let $\BA=[0,1]\times S^{1}$ be the annulus, with standard area form $\omega_0$ and fixed primitive form $\beta_0$, let $f$ be an area-preserving, orientation preserving diffeomorphism on $\BA$ that is pure rotations of $\theta_i$ near the boudary $A_{i}, i=0,1$. Then we have 
\begin{itemize}

\item If $F< Cal(f)$, then there exists $\gamma_{0}$, $\mathcal{A}(\gamma_{0})\geq Cal(f)$;
\item If $\theta_{0}< Cal(f)$, then there exists $\gamma_{1}$, $\mathcal{A}(\gamma_{1})\geq \frac{1}{2}(Cal(f)+F+\theta_{0})-\rho(\gamma_{a})$;
\item If $\theta_{0} \leq F < Cal(f) $, then  there exists $\gamma_{\infty}$, $\rho(\gamma_{\infty})\geq\theta_{0}$; and if $\rho(\gamma_{\infty})=\theta_{0}$ then $\mathcal{A}(\gamma_{\infty})\geq 2F-\theta_{0}$.

\end{itemize}

\end{theorem}

This article is orgainzed as follow. In Section 2 we recall the basic definitions of action functions and rotation numbers for invariat measures of area-preserving surface diffeomorphism, and discuss their relationship in detail. In Section 3 we analyze the pricise change of the action function in an augmented system, prove the main theorem, and provide an insightful remark on the existence of periodic orbits based on the information of actions and rotations. in Section 4, we present some fundamental examples of area-preserving surface diffeomorphism, for which we could plot the bivariate diagram of invariant measures.  

\section{Action function and rotation numbers}

In this section, we give a general introduction to the action and
rotation vectors of symplectic diffeomorphisms. 

\subsection{Action function}

Let $M$ be a $2n$-dimensional symplectic manifold with a symplectic form
$\omega$ and non-empty boundary, let $f$ be an exact symplectic
diffeomorphism on $M$, i.e., $f$ is isotopic to identity, and for any
1-form $\beta$ such that $\omega = d \beta$, we have
$$ f^{*}\beta-\beta=dg $$

is exact. The real-valued function $g$ on $M$ is called the {\it action function} of $f$.

Clearly, the choice of $g$ is not unique, it depends on two factors, the first one is the integration constant. This can be fixed by assigning $g$ a particular value at a special point, for example on the boundary.  Hutchings and Weiler assumed the action function value on the circle boundary so that the value aggrees with the circle rotation number of the diffeomorphism.

The second factor is the choice of $\beta$ we made. For fixed $\beta$, we define the {\it mean action} of $f$ as: 
\begin{definition}
The mean action of $f$ is defined by
$$\mathcal{A}(f,\beta)=\frac{\int_{M}g\omega^{n}}{\int_{M}\omega^{n}}$$
we may rescale $\omega$ so that $\int_{M}\omega^{n}=1$. 
\end{definition}

If $M$ is the unit disk $D^2$, then we will see later that after fix some boundary condition, the choice of the 1-form makes no difference, the mean action is called the {\it Calabi invariant}.


The definition of action with respect to the volume form can be extended to all $f$-invariant finite measures on $M$. More precisely, let $\mathcal{M}_{f}$ be the set of all $f$-invariant finite measures on $M$, For any $\mu\in\mathcal{M}_{f}$ and a fixed closed 1-form $\beta$, the mean action of $\mu$ is defined as 
$$\mathcal{A}(f,\mu,\beta) = \frac{\int_{M} g d\mu}{\int_{M}d\mu}.$$
Sometimes we omit $f$ while there is no danger of confusing and we want to emphasize the difference in the invariant measure.

In particular, for any periodic orbit $\gamma = \{p_{0}, p_{1},\ldots, p_{k} = p_{0}\}$, where $f^{i}(p_{0}) = p_{i}$, for $ i = 1, 2,..., k $, the corresponding probability invariant measure is
$$\mu_{\gamma} = \frac{1}{k}(\delta_{p_{0}} +\delta_{p_{1}} +\ldots + \delta_{p_{k-1}}),$$
and the mean action on $\gamma$ is
$$\mathcal{A}(\gamma,\beta)=\mathcal{A}(\mu_{\gamma},\beta) = \frac{1}{k}(g(p_{0}) + g(p_{1}) + \ldots + g(p_{k-1})).$$

To see the dependence of the mean action on the closed 1-form, notice that if
$$\omega = d\beta = d\beta', $$ 
then $\beta'-\beta$ is a closed form. We will see later the cohomology class of this closed form plays an essential role. 

\begin{lemma} If  $[\beta'-\beta] = 0\in H^{1}(M, \R)$ then this two
  forms determinate the same mean action, up to a constant.
\end{lemma}

One can see \cite{Deng21}  for the proof.  The case that $[\beta'-\beta]\neq0$ is more interesting and is related to the rotation vector of $f$-invariant finite measures, we will give more datiles in next section.  

The following example assert the geometry meaing of action function.

\begin{example}
Let $(M,\omega)$ be a surface, $x_{0}$ be a fixed point of $f$, and let $\gamma$ be any smooth curve that from $x_{0}$ to $x$. Then we have
$$\int_{\gamma}f^{*}\beta-\beta=\int_{f(\gamma)-\gamma}\beta=g(x)-g(x_{0}),$$
it does not depend on the choice of $\gamma$ we made.

If both $x_{0}$ and $x$ are fixed points of f, then 
$$g(x)-g(x_{0})=\int_{U}\omega,$$ 
where $U$ is a disk such that $$\partial D=f(\gamma)-\gamma.$$
\end{example}

The following lemma and example assert that if we perturb a differmorphism with a diffeomorphism that admits nontrivial mean action, then we could get new periodic orbits according the conclusion of previous theorems.

\begin{lemma} Let $f_{1},f_{2}$ be two exact symplectic diffeomorphisms on $(M,\omega)$. Let $g_{1},g_{2},g_{12}$ be action functions for $f_{1},f_{2}$ and $f_{2}\circ f_{1}$ respectively. Suppose there is a point $x_{0}\in M$ such that 
\begin{equation*}
g_{12}(x_{0})=g_{1}(x_{0})+g_{2}(f_{1}(x_{0}))\label{1},
\end{equation*}
then $$\mathcal{A}(f_{2}\circ f_{1})=\mathcal{A}(f_{2})+\mathcal{A}(f_{1}).$$ 
\end{lemma}
Notice the requirement about the value of $g$ can be satisfied by adding some constant on each action function. The proof can be found in \cite{Deng21}.

\begin{example}
Assume $y_{0}\in\mathbb{R}-\mathbb{Q}$, let $f:\mathbb{A}\rightarrow\mathbb{A}, (x,y)\mapsto(x,y+y_{0})$ be an irritional rotation on annulus.
Then for the standard area form $\omega$ and the 1-form $\beta=ydx$, the action function of f is constant $y_{0}$ everywhere, and $\mathcal{A}(f)=y_{0}$. Obviously $f$ admits no periodic orbit on $\BA$. This is a sharp count example of the theorem.

Consider the Hamiltonian diffeomorphism $f'$ on $\mathbb{A}$ generated by
$$H(x,y)= \left\{ 
\begin{aligned} &\ y_{0}-x^2-y^2,  &\ x^2+y^2<y_{0}   \\
                      &\ \ \ 0,  &\ x^2+y^2>y_{0}+\varepsilon \end{aligned} \right.$$
and smooth and monotone between.  $f'$ is a clockwise rotation in the area $x^2+y^2<y_{0}$ and have negative mean action, hence  $\mathcal{A}(f\circ f')<\mathcal{A}(f)=y_{0}$ and $f\circ f'$ has periodic points by our main theorem.
\end{example}

\subsection{Rotation numbers}

Let $f$ be an exact symplectic diffeomorphism on a $2n$ dimensional
symplectic manifold $M$, then the measures support on periodic orbits of
$f$ are $f$-invariant. For a periodic orbit
$\gamma = \{p_0, p_1, \ldots, p_k=p_0\}$, let $\tilde{\gamma}(t)$ be a
closed curve on $M$ obtained by isotopy, with
$\tilde{\gamma}(i) = p_i$, for $i = 1, 2,\ldots, k$.  We remark here
we have made a choice of the isotropy of $f$, i.e., a path $f^{t}$ of
area-preserving diffeomorphisms such that $f^{0}=id$ and $f^{1}=f$,
$\tilde{\gamma}$ is the curve on $M$ defined by
$\tilde{\gamma}(t,x)=f^{t}(x)$. The following definitions do not
depend on the choice of isotropy.
\begin{definition}
The rotation vector of $\gamma$ is the homology class of $\tilde{\gamma}$ divided by its period,
$$\rho(\gamma) = \frac{1}{k} [\tilde{\gamma} ] \in H_1(M, \R).$$
\end{definition}

To generalize the concept of rotation vector, for any closed 1-form
$\alpha$ on $M$, we define the bi-linear form $\langle\cdot, \cdot\rangle^*$ by 
$$\langle\gamma, \alpha\rangle^* = \frac{1}{k}\oint_{\tilde{\gamma}} \alpha.$$
It is easy to see that the above pairing depends only on the homology
class of $\tilde{\gamma}$, in $H_1(M, \R)$, and the cohomology class
of $\alpha$, in $H^1(M, \R)$. This paring equivalently defines, for
any periodic orbit $\gamma$, the rotation vector
$\rho(\gamma) \in H_1(M, \R)$, by the following equation:
$$\langle\rho(\gamma), [\alpha]\rangle = \langle\gamma, \alpha\rangle^* =
\frac{1}{k}\oint_{\tilde{\gamma}} \alpha,$$ where the first pairing is
the canonical pairing between homology and cohomology of the manifold
$M$.

This definition of rotation vector can be naturally extended to $f$-invariant finite
measures in $\mathcal{M}_f$. Fix a closed 1-form $\alpha$, for any point $x\in M$, let $\tilde{\gamma}(t, x), t\in\R$ be the curve in $M$ by connecting
orbit of $x$ by the isotopy in such a way that $\tilde{\gamma}(i, x) =
f^i(x)$. Define, if exists,
$$\rho(x, \alpha) = \lim_{T \rightarrow\infty} \frac{1}{T} \int_{\tilde{\gamma}(t, x): t \in [0,T]} \alpha.$$
It is easy to see that $\rho(x, \alpha) = \rho(x, \alpha')$, if
$[\alpha] = [\alpha'] \in H^1(M, \R)$.  This is because if
$\alpha$ is exact, $\alpha = dS$ for some 
function $S: M \rightarrow \R$, then 
$$\rho(x, dS) =\lim_{T \rightarrow\infty} \frac{1}{T} \int_{\gamma(t, x): t \in [0,T]} dS = \lim_{T \rightarrow \infty} \frac{1}{T} (S(\gamma(T)) - S(x)) = 0.$$

For any invariant finite measure $\mu \in \mathcal{M}_f$, by Birkhoff Ergodic Theorem, for any fixed $\alpha$, for $\mu-a.e. \, x \in M$, the limit exists and $\rho(x, \alpha)$ is well-defined. Therefore, it is also well-defined for a finite set of basis vectors $[\alpha]$ in $H^1(M, \R)$. Hence, for $\mu -a.e. \, x \in M$,
$\rho(x, \alpha)$ is well-defined for {\em all}\/ closed 1-forms $\alpha$. Moreover, by Birkhoff Ergodic Theorem,
$$\int \rho(x, \alpha) d\mu = \int \left(\int_{\tilde{\gamma}(t, x): t \in [0, 1]}
  \alpha \right) d \mu$$
The above equation is linear in both $\alpha$ and $\mu$, it depends only on the cohomology class of $\alpha$, therefore, it defines a pairing between $\mu$ and cohomology elements in $H^1(M, \R)$. This
pairing defines the rotation vector  $\rho(\mu) \in H_1(M, \R)$.
\begin{definition}
For any $\mu \in \M_f$ and any closed 1-form $\alpha$, $[\alpha] \in
H^1(M, \R)$, the {\em rotation vector} of $\mu$, $\rho(\mu) \in H_1(M, \R)$ is defined
by the following equation  
$$\langle\rho(\mu), [\alpha]\rangle = \int \left(\int_{\tilde{\gamma}(t, x): t \in [0, 1]}
  \alpha \right)
d \mu \in \R$$
where the left-hand side is
the canonical pairing between homology and cohomology of the manifold
$M$.
\end{definition}

As an example, if $f$ is a Hamiltonian diffeomorphism, then $\omega^n \in \mathcal{M}_f$ and if $M$ is compact without boundary, then $\rho(\omega^n) = 0$, where $\omega^n$ denote the f-invariant measure determined by the volume form $\omega^n$.  This is not true in general for non-Hamiltonian symplectic diffeomorphisms. An easy example is 
$$T_{(a, b)}: \T^2 \rightarrow \T^2,\ \ T_{(a, b)} (x, \, y) = (x + a, \, y + b) \mbox{ mod } \Z^2 (a, b) \notin \Z^2$$ This is also not true in general for Hamiltonian diffeomorphisms on manifold with boundaries, for example the annulus.

Recall that the action we defined depends on the choice of 1-form
$\beta$. It turns out that this dependence is closely related to the rotation vector we just defined. 
\begin{proposition}
Let $\alpha$ be a non-trivial closed 1-form on $M$, let $\tilde{\beta} = \beta + \alpha$ satisfying $d\tilde{\beta} = d\beta = \omega$. Let $\tilde{g}$ and $g$ be action functions defined by $\tilde{\beta}$ and $\beta$ respectively. To fix the integration constants, pick a point $x_0 \in M$,  fix the value of $g(x_{0})=\tilde{g}(x_{0})=0$, then we have
$$\mathcal{A}(\mu,\tilde{\beta}) - \mathcal{A}(\mu,\beta) = \langle\rho(\mu), \alpha\rangle - C_{x_0, \alpha}$$
where $C_{x_0, \alpha}$ is a constant depending on $x_0$ and $\alpha$, 
but not on $\mu$.
\end{proposition}

\begin{proof}
By definition, for any $x$, ${g}(x)  = \int_{x_0}^x ( f^*({\beta}) - \beta)$
and likewise $$\tilde{g}(x)  = \int_{x_0}^x ( f^*(\tilde{\beta}) -
\tilde{\beta)} = g(x) + \int_{x_0}^x (f^* \alpha - \alpha).$$
All the integrals are path independent.  To see this notice if we pick two different curves $l_{1}$ and $l_2$ that links from $x_0$ to $x$, then 
\be \int_{l_1}f^{*}\beta-\beta-\int_{l_2}f^{*}\beta-\beta=\oint f^{*}\beta-\beta=\int_{U}d(f^{*}\beta-\beta) =\int_{U}ddg=0. \nn \ee
where the circle intergration is oriented by $l_{1}-l_{2}$, and $U$ is the closed area bounded by $l_{1}$ and $l_2$. $\tilde{g}$ case is similar.

For any invariant finite measure $\mu$, by Stokes' theorem, we have 

\begin{eqnarray*}
\mathcal{A}(\mu,\tilde{\beta}) &=& \int_M \tilde{g} d\mu = \int_M g d\mu + \int_M \left(
\int_{x_0}^x (f^* \alpha - \alpha) \right) d\mu  \\
&=& \mathcal{A}(\mu,\beta) + \int_M \left( \int_{\gamma(t, x): t \in [0, 1]} \alpha
\right) d\mu - \int_M \left( \int_{\gamma(t, x_0): t \in [0, 1]} \alpha
\right) d\mu  \\
&=& \mathcal{A}(\mu,\beta) + \langle\rho(\mu), \alpha\rangle - C_{x_0, \alpha}.
\end{eqnarray*}

\end{proof}

The constant $C_{x_0, \alpha}$ is zero, if our reference point $x_0$ is a contractible fixed point.

\section{Main result}

Now we consider the diffeomorphisms on Annulus. Let $\mathbb{A}=[0,1]_{x}\times\mathbb{R}/\mathbb{Z}_{y}$ be the annulus with the standard area form $\omega_{0}=dxdy$, and fix a primitive $\beta_{0}=xdy$ such that $d\beta_{0}=\omega_{0}.$ let $f$ be an area-preserving, orientation preserving diffeomorphism on $\mathbb{A}$ that is rotation of $\theta_{1}$, $\theta_0$ near $\partial\mathbb{A}=A_{1}\cup A_{0}$ respecitively. That is, there are  
$\theta_{1}, \theta_0 \in\mathbb{R}$ and a lift $\tilde{f}$ of $f$ to the universal cover $\tilde{\BA}=[0,1]\times \R$ of $\BA$,  such that 
\EQN{\tilde{f}(x,y)=\begin{cases}(x, y+\theta_1)& \quad \text{for $x$ sufficiently close to 1,}\\
(x, y+\theta_{0})& \quad \text{for $x$ sufficiently close to 0.}\end{cases}}

\begin{definition}

The {\it flux} of $f$, denoted as $F$, is the area, in $\tilde{\BA}$ between $f(l)$ and $l$, where $[l]$ is the generator of $H_1(\BA,\partial\BA)$ oriented from $\BA_{0}$ to $\BA_1$, with respect to the standard area form $\tilde{\omega}$ on $\tilde{\BA}$.

\end{definition}

There is a natural relationship between the flux and the rotation
vector, see \cite{Qu24} for the proof. 

\begin{lemma}
Let $F$ be the flux of $\tilde{f}$, let $\alpha=dy$, notice $[\alpha] \in
H^{1}(\BA, \R)$,  we have: 
$$F=\langle\rho(\omega), [\alpha]\rangle$$
here $\rho(\omega)$ is the rotation vector of $f$ with respect to the area
$\omega$.
\end{lemma}

In order to apply theorem \ref{Hut} or theorem \ref{shangxia}, it is natural trying to fill the interior hole of $\BA$ to get a disk $D$. 

More precisely, use the standard polar coordiante $D=(r,\theta)$ on $D$, let $\omega_{D}=2rdrd\theta$ then it is a standard area form such that $\int_{D}\omega_{D}=1$, and $\beta_{D}=r^{2}d\theta$ is a primitive form of $\omega_{D}$.

let $i_{a}: \mathbb{A}\to D $ be the immersion that send $\BA$ to the out stip of $D$ while 
$$i_{a}(A_{0})[(0,y)]=\{r=\sqrt{\frac{a}{1+a}}\}\subset D,\ i_{a}(A_{1})=\partial D.$$ 
For example we may assume that 
$$i_{a}(x,y)=(\sqrt{\frac{x}{a+1}+\frac{a}{a+1}}, y)=(r,\theta),$$
then
$i_{0}^{*}\omega_{D}=\omega_{0}$, $i_{0}^{*}\beta_{D}=\beta_{0}$, and $i_{a}^{*}\omega_{D}=\frac{1}{a+1}\omega_{0}$,
$$ i_{a}^{*}\beta_{D} = \frac{1}{(a+1)}\beta_{0} +  \frac{a}{(a+1)}dy. $$
In this way, we get a family of different diffeomorphisms on $D$, denote as $$f_{a} = i_{a}\circ f \circ i_{a}^{-1}. $$ 

Recall that $f$ is rigid rotations near $\partial A$ with rotation number $\theta_{1}$ and $\theta_{0}$. The natural extension for $f$ to the whole disk is by letting $f$ to be rigid rotation with rotation number $\theta_{0}$' in the interior hole, in which the origional point $(0,0)\in D $ is an elliptic fixed point of the extended diffeomorphism. We use the same notion $f_{a}$ to denote the extended diffeomorphisms that is rigid rotation in the interior hole.



In order to calculate the Calabi invariant and mean action of $f_a$,  on $\BA$ let $$i_{a}^{*}\beta_{D} =\beta_{a} = \frac{1}{a+1}(\beta_{0}+ady).$$ 

We see in section 2 that if $\beta_{0}$ and $\beta_{a}$ are two primitive of $\omega$ on $\BA$ such that $\beta_{0}-\beta_{a}\in H^{1}(\BA)$ is non-trivial, then the corresponding Calabi invariant and mean action of periodic orbits do not concide.  Let $g_{0}$ and $g_{a}$ be the action function for $\beta_{0}$ and $\beta_{a}$ respectively and fix $g_{a}(A_{1})=g_{0}(A_{1})=\theta_{1}$. If $\beta_{a}-\beta_{0}=\alpha$, by definition, for any $x\in\BA$, 
\be\label{delta} g_{a}(x)-g_{0}(x)=\int_{l}f^{*}\alpha-\alpha \ee 
where $l$ is any path connecting x with $A_{1}$.

Let $g_{a}\in C^{\infty}(D,\R), a\in [0,\infty)$ be the action function of $f_{a}$, denote the origional point of D as $(0,0)$, fix $g_{0}(\partial D)=g_{a}(\partial D)=\theta_{1}$,  we have the follow results.

\begin{proposition}
\begin{itemize}
\item Denote $F$ to be the flux of $f$ on $\BA$,  we have
\be g_{0}(0,0) = F; g_{a}(0,0)=\frac{F}{1+a}+\frac{a\theta_{0}}{1+a}. \ee 
\item The Calabi invariant of $f_a$ satisfyies
\be Cal(f_{a}) = \frac{Cal(f)}{(1+a)^2}+\frac{2aF}{(1+a)^2}+\frac{a^{2}\theta_{0}}{(1+a)^2}. \ee
\item For a periodic orbit $\gamma_{a}$ of $f_{a}$ that correspondence to $\gamma\in\mathcal{P}(f)$ on $\BA$ with rotation number $\rho(\gamma)$, it holds
\be \A(f_{a}, \gamma_{a})=\frac{\A(f, \gamma)}{1+a}+\frac{a}{1+a}\rho(\gamma).\ee

\end{itemize}

\end{proposition}

\begin{proof}

Let  $l$ be any curve that link $\partial D$ and $(0,0)$. By definition, 

\begin{eqnarray*}
g_{0}(0,0) - g_{0}(\partial D) &=& \int_l dg_{0} =\int_{l} (f_{0}^{*}\beta_{D}-\beta_{D}) \\
&=&\int_{i^{-1}_{0}(l)} (i_{0}^{*} (i_{0}^{-1}\circ f\circ i_{0})^{*}\beta_{D}-i_{0}^{*}\beta_{D}) \\
&=&\int_{i^{-1}_{0}(l)} (f^{*}i_{0}^{*}\beta_{D}-i_{0}^{*}\beta_{D}) \\
&=&\int_{f(l')-l'}\beta_{0} \\
&=&\int_{V} d\beta_{0}-\int_{b}\beta_{0}=F-\theta_{1} 
\end{eqnarray*}

where $i_{0}^{-1}(l)=l'$, $b$ denote a part of $\partial \BA$ between $l'$ and $f(l')$, $V$ is the area bounded by $l', f(l')$ and $b$. 

Similar, let $l$ be any curve that link $\partial D$ and the little circle $r=\frac{a}{a+1}$ in $D$, we have  

\begin{eqnarray*}
g_{a}(0,0) - g_{a}(\partial D) &=& \int_l dg_{a} =\int_{l} (f_{a}^{*}\beta_{D}-\beta_{D}) \\
&=&\int_{i^{-1}_{a}(l)} (i_{a}^{*} (i_{a}^{-1}\circ f\circ i_{a})^{*}\beta_{D}-i_{a}^{*}\beta_{D}) \\
&=&\int_{i^{-1}_{a}(l)} (f^{*}i_{a}^{*}\beta_{D}-i_{a}^{*}\beta_{D}) \\
&=&\int_{f(l')-l'}\beta_{a} \\
&=&\int_{V} d\beta_{a}-\int_{b}\beta_{a} \\
&=&\frac{1}{a+1}(F-\theta_{1})-\frac{a}{a+1}(\theta_{1}-\theta_{0}).
\end{eqnarray*}

And notice that $g_{a}$ is a constant in the interior hole, we have
$$ g_{a}(0,0)=\frac{1}{a+1}F  + \frac{a}{a+1} \theta_{0}.  $$
To calculate $Cal(f_{a})$, denote the stip part as $D_{a}=i_{a}(\BA)\subset D$ and the inside hole as $d_{a} = D\setminus D_{a}$, we have
$$\int_{D_{a}}g_{a}\omega_{D} = \int_{i_{a}(\BA)} g_{a} d\beta_{D} = 
\int_{\BA}i_{a}^{*}(g_{a} d\beta_{D}) = \int_{\BA} (g_{a}\circ i_{a})d(i_{a}^{*}\beta_{D});
$$
That is the Calabi invariant of $f$ on $\BA$ using $\beta_{a}$ and $d\beta_{a}$, hence we have 
$$\int_{D_{a}}g_{a}\omega_{D} = Cal(f,\beta_{a}) =\frac{1}{a+1}(Cal(f,\beta_{0}) + \frac{a}{a+1}F). $$

One can find more details in \cite{Qu24} about this calculation. On the other hand, $g_{a}$ is a constant on $d_a$. By definition, 
\begin{eqnarray*}
Cal(f_{a}) & = &\frac{1}{a+1} \int_{D_{a}}g_{a}\omega_{D} + \int_{d_{a}}g_{a}\omega_{D} \\
& = &\frac{1}{a+1} \int_{D_{a}}g_{a}\omega_{D} + \frac{a}{a+1} g_{a}(0,0) \\
& = &\frac{Cal(f)}{(1+a)^2}+\frac{2aF}{(1+a)^2}+\frac{a^{2}\theta_{0}}{(1+a)^2}.
\end{eqnarray*}


To proof (3.4), notice that 

\begin{eqnarray*}
\A(f_{a}, \gamma_{a})-\frac{\A(f, \gamma)}{1+a} &=& \A(f,\gamma_{a},\beta_{a})-\A(f,\gamma,\frac{1}{1+a}\beta_{0})\\
&=& \int_{\BA}(\int_{\gamma}f^{*}\frac{a}{1+a}dy-\frac{a}{1+a}dy) \\
&=&\frac{a}{1+a}\int_{\BA}(\int_{\gamma}df_{y}-dy)d\mu_{\gamma}\\
&=&\frac{a}{1+a}\rho(\gamma)
\end{eqnarray*}

the second equality is because $\beta_{a}-\frac{1}{1+a}\beta_{0}=\frac{a}{1+a}dy$.

\end{proof}

Recall our main theorem is:

\begin{theorem}(Theorem 1.4)

Let $\BA=[0,1]\times S^{1}$ be the annulus, with standard area form $\omega_0$ and fixed primitive form $\beta_0$, let $f$ be an area-preserving, orientation preserving diffeomorphism on $\BA$ that is pure rotations of $\theta_i$ near the boudary $A_{i}, i=0,1$. Then we have 
\begin{itemize}

\item If $F< Cal(f)$, then there exists $\gamma_{0}$, $\mathcal{A}(\gamma_{0})\geq Cal(f)$;
\item If $\theta_{0}< Cal(f)$, then there exists $\gamma_{1}$, $\mathcal{A}(\gamma_{1})\geq \frac{1}{2}(Cal(f)+F+\theta_{0})-\rho(\gamma_{a})$;
\item If $\theta_{0} \leq F < Cal(f) $, then  there exists $\gamma_{\infty}$, $\rho(\gamma_{\infty})\geq\theta_{0}$; and if $\rho(\gamma_{\infty})=\theta_{0}$ then $\mathcal{A}(\gamma_{\infty})\geq 2F-\theta_{0}$.

\end{itemize}

\end{theorem}

\begin{proof}

If  $g_{a}(0,0) < Cal(f_{a})$,  by applying Theorem \ref{shangxia} for $f_{a}$ on $D$, there exists a periodic orbit $\gamma_{a}$ other than $(0,0)$ such that 
\be\label{fa} \A(f_{a}, \gamma_{a})\geq Cal(f,\beta_{a}) \ee

$g_{a}(0,0) < Cal(f_{a})$ is equivalent to \be (1-a)F + a \theta_{0} < Cal(f) \ee


\ref{fa} is equivalent to 
$$\frac{\A(f, \gamma_{a})}{1+a}+\rho(\gamma_{a})\frac{a}{1+a}\geq\frac{Cal(f)}{(1+a)^2}+\frac{2aF}{(1+a)^2}+\frac{a^{2}\theta_{0}}{(1+a)^2} $$  which imply
\be
\A(f,\gamma)&\geq&\frac{1}{1+a}(Cal(f) +2aF+a^{2}\theta_{0}))-a\rho(\gamma_{a})  \\
&=&\frac{1}{1+a}Cal(f)+\frac{a}{1+a}(2F-\theta_{0})+a(\theta_{0}-\rho(\gamma_{a})).
\ee

Hence, if $(1-a)F + a \theta_{0} < Cal(f) $ for some $a\in\R_{\geq0}$, then there exist $\gamma_{a}\in\mathcal{P}(f)$ satisfying (3.8). In particular, $a=0, a=1, a=\infty$ give the results of theorem.

\end{proof}

\begin{remark}
In \cite{Wei21} Weiler used a different coordinate such that the total area of $\BA$ is 2, and the area form has positive part and negitave part.
\end{remark}

\

\begin{remark}

We can give more geometric explaination for this theorem. As in the following picture, let x-axis and y-axis denote the rotation number and mean action of a f-invariant measure respectively, then all the point of invariant measures is a convex set, and the red point is the f-invariant measure induced by the area form $\omega$.  


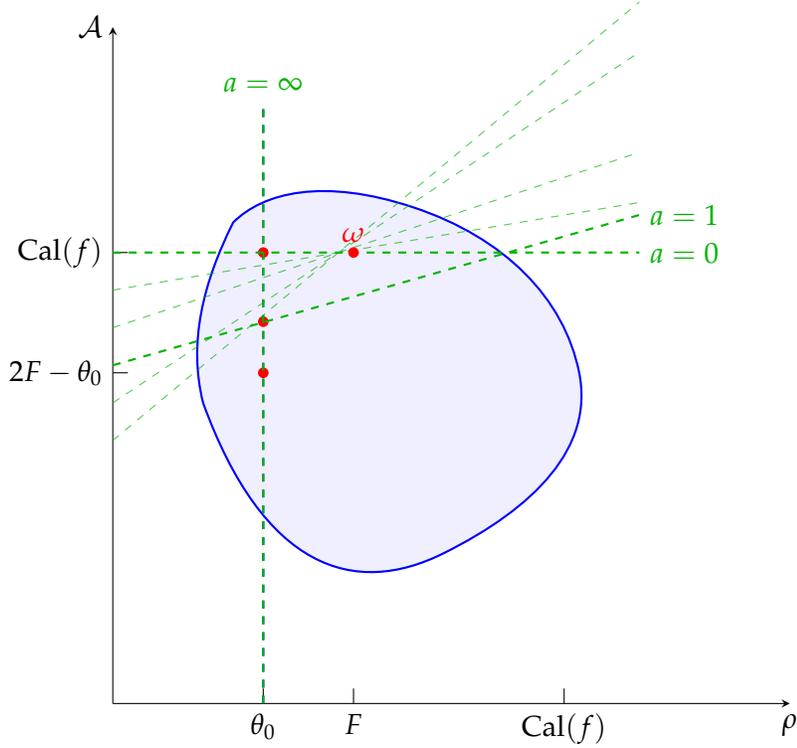
\begin{tikzpicture}[scale=2, >=stealth]

\draw[->] (0,0) -- (4.5,0) node[below] {$\rho$};
\draw[->] (0,0) -- (0,4.5) node[left] {$\mathcal{A}$};

\coordinate (theta0) at (1,0);
\coordinate (F) at (1.6,0);
\coordinate (Calf) at (3,0);

\draw (theta0) node[below] {$\theta_0$} -- ++(0,0.1);
\draw (F) node[below] {$F$} -- ++(0,0.1);
\draw (Calf) node[below] {$\text{Cal}(f)$} -- ++(0,0.1);

\coordinate (CalfY) at (0,3);
\coordinate (twoFtheta0) at (0,2.2);

\draw (CalfY) node[left] {$\text{Cal}(f)$} -- ++(0.1,0);
\draw (twoFtheta0) node[left] {$2F-\theta_0$} -- ++(0.1,0);

\draw[dashed, blue, thick] (theta0) -- ($(theta0)+(0,4)$);

\draw[blue, fill=blue!20, fill opacity=0.3, thick, smooth cycle] 
    (0.6,2.0) .. controls (0.5,2.4) and (0.6,2.8) .. 
    (0.8,3.2) .. controls (1.0,3.4) and (1.4,3.5) .. 
    (2.0,3.3) .. controls (2.6,3.1) and (3.0,2.7) .. 
    (3.1,2.2) .. controls (3.2,1.7) and (2.8,1.3) .. 
    (2.2,1.0) .. controls (1.6,0.7) and (1.0,0.9) .. 
    (0.6,2.0);
    
\fill[red] (1,3) circle (1pt);      
\fill[red] (1,2.54) circle (1pt);   
\fill[red] (1,2.2) circle (1pt);    

\draw[green!70!black, dashed, thick] (0,3) -- (3.5,3) 
    node[right] {$a=0$};

\draw[green!70!black, dashed, thick] (0,2.25) -- (3.5,3.25)
    node[right] {$a=1$};

\draw[green!70!black, dashed, thick] (1,0) -- (1,4)
    node[above] {$a=\infty$};

\foreach \a in {0.2,0.5,2,5} {
    \pgfmathsetmacro{\intercept}{3/(1+\a) + 1.5*\a/(1+\a)}
    \pgfmathsetmacro{\slope}{\a/(1+\a)}
    \draw[green!70!black, dashed, opacity=0.7] 
        (0,\intercept) -- (3.5,\intercept+\slope*3.5);
}

\fill[red] (F) ++(0,3) circle (1pt) 
    node[above] {$\omega$};
    

\end{tikzpicture}


Assume $\theta_{0}\leq F < Cal(f)$, then for any $a\geq 0$, we have $g_{a}(0,0) < Cal(f_{a})$. The key fact is that for any $a\geq0$ there exist a invariant measure support on periodic orbit $\gamma_{a}$ over the line $\frac{1}{1+a}Cal(f)+\frac{a}{1+a}(2F-\theta_{0})+a(\theta_{0}-\rho(\gamma_{a}))$, the green lines in the picture, $a=0$ is the horizontal line, $a\to\infty$ is the vertical line and as $a$ varies from 0 to $\infty$, the slope of the green line changes continuously, and its intersection with the vertical line $\theta_{0}=\rho(\gamma_{a})$ gradually shifts from $Cal(f)$ to $2F-\theta_{0}$.

If $f$ admits finintely many periodic orbits then there must exists $\gamma$ in the overlap area of infinitely many green lines, which means there exists $\gamma$ satisfying
 $$\mathcal{A}(\gamma)\geq\max\{Cal(f),2F-\theta_{0}\}, \rho(\gamma)\leq\theta_{0}.$$
 
      For other situations like $F\leq \theta_{0}< Cal(f)$, or $\theta_{0}\geq F > Cal(f)$ , $F\geq\theta_{0} > Cal(f)$,  we could derive similar conclusions through parallel reasoning.

\end{remark}

\section{Examples}

\begin{example}[twist map]

\

Consider annunus diffeomorphism $f(x,y)=(x,y+a_{1}x+a_{0})$, where $a_{1}, a_{0}\in\mathbb{R}$. Then $y_{1}=a_{1}+a_{0}, \ y_{-1}=-a_{1}+a_{0}, \ F=2a_{0},$
and $$g(x,y)=\frac{a_{1}}{2}x^{2}+\frac{a_{1}}{2}+a_{0}$$

Or similar, consider disk diffeomorphism $f(r,\theta)=(r,\theta+a_{1}r+a_{0})$ on $D^2=[0,1]\times S^1$. Then $y=a_{1}+a_{0}$, and 
 $$g(r,\theta)=\frac{a_{1}}{2}r^{2}+\frac{a_{1}}{2}+a_{0}$$

In this case, the f-invariant measures are circles and annulus centered at the origin, and linear combinations of them. The graph of of rotation number and mean action of invariant measures acts as follow.

\end{example}


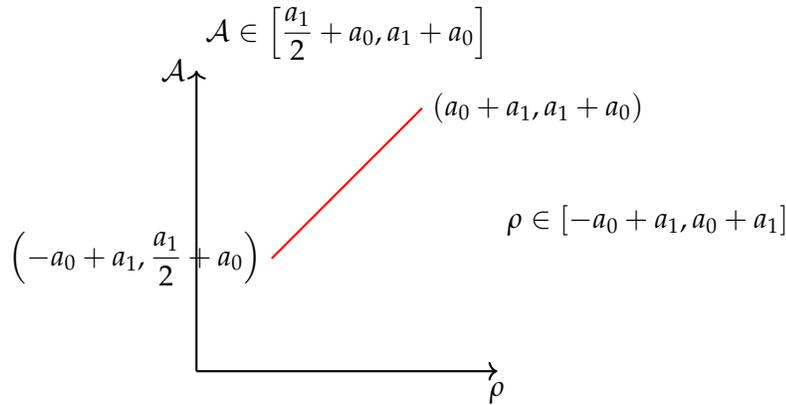
\begin{tikzpicture}

\draw[->, thick] (0,0) -- (4,0) node[below] {$\rho$};
\draw[->, thick] (0,0) -- (0,4) node[left] {$\mathcal{A}$};

\draw[red, thick] (1,1.5) -- (3,3.5);
\node[left] at (1,1.5) {$\left(-a_0+a_1, \dfrac{a_1}{2}+a_0\right)$};
\node[right] at (3,3.5) {$\left(a_0+a_1, a_1+a_0\right)$};

\node[right] at (4,2) {$\rho \in [-a_0+a_1, a_0+a_1]$};
\node[above] at (2,4) {$\mathcal{A} \in \left[\dfrac{a_1}{2}+a_0, a_1+a_0\right]$};

\end{tikzpicture}


If the projection of the graph of rotation number and mean action of invariant measures onto the x-axis is an interval, then according to Franks' theorem there must exist infinitely many periodic points. Here we give some examples beyond this situations.

\begin{example}

The pure irrational rotation on disk or annulus, the graph of rotation number and mean action of invariant measures is a single point on the first quadrant. In \cite{Le23}, Patrice Le Calvez gives a finite dimensional analysis for the irrational pseudo-rotation case on disk, that is the area-preserving disk diffeomorphisms that admit only one fixed point and no periodic points, in which the action is a constant irrational number.

On the other hand, if the mean action-rotation number graph of an area-preserving disk diffeomorpism $f$ is a single irrational point like shown in the following picture, we can prove that f is an irrational pseudo-rotation.
\end{example}


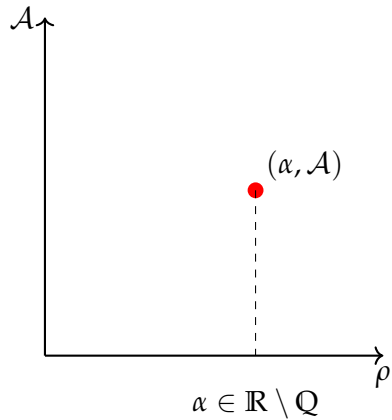
\begin{tikzpicture}

\draw[->, thick] (0,0) -- (4.5,0) node[below] {$\rho$};
\draw[->, thick] (0,0) -- (0,4.5) node[left] {$\mathcal{A}$};

\fill[red] (2.8,2.2) circle (3pt);
\node[above right] at (2.8,2.2) {$(\alpha, \mathcal{A})$};

\draw[dashed] (2.8,2.2) -- (2.8,0);
\node[below] at (2.8,-0.3) {$\alpha \in \mathbb{R}\setminus\mathbb{Q}$};

\end{tikzpicture}

\begin{example}

We construct an example in which the projection of the graph the graph of rotation number and mean action of invariant measures onto the x-axis is a single point, onto the y-axis is an interval. Thanks to David Bechara Senior for bring this example to my attention.

The idea of the construction is the following. Consider a map $\phi=\phi_{0}\circ\phi_{1}$ consists of a composition of two area-preserving diffeomorphisms on the disk $D$, where $\phi_{0}$ is pure rotation clockwisely of , say $1/4$ on the disk $D$, and $\phi_{1}$ is a diffeomorphism that supported on four small disks that symmetrically distributed with respect to the origin of $D$, on which small disk $\phi_{1}$ is a twist rotation that rotate by $1/4$ near the boundary of the small disk and rotate faster (or slower) near the origin of the small disk.  Then the rotation number of $\phi$-invariant measures are a constant number $1/4$ and the mean action is a integral that represents the rotational speed of  $\phi_{1}$ on each small disk, as shown in the follwoing picture.

\end{example}

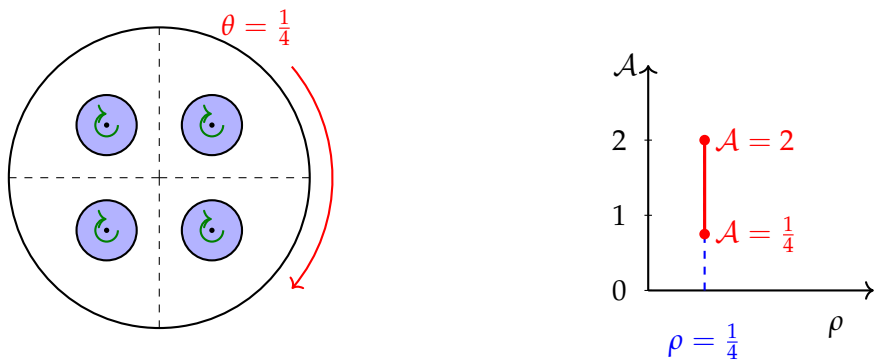
\begin{tikzpicture}

\begin{scope}[xshift=-3cm]
\draw[thick] (0,0) circle (2cm);

\foreach \angle in {0,90,180,270} {
    \draw[dashed] (0,0) -- (\angle:2);
}

\foreach \x/\y in {0.7/0.7, -0.7/0.7, -0.7/-0.7, 0.7/-0.7} {
    \fill[blue!30] (\x,\y) circle (0.4cm);
    \draw[thick] (\x,\y) circle (0.4cm);
    \fill[black] (\x,\y) circle (1pt);
}

\draw[red, thick, ->] (40:2.3) arc (40:-40:2.3);
\node[red] at (1.3,2.0) {$\theta = \frac{1}{4}$};

\foreach \x/\y in {0.7/0.7, -0.7/0.7, -0.7/-0.7, 0.7/-0.7} {
    \draw[green!50!black, thick, ->] 
        ([shift={(0:0.15)}] \x,\y) arc (0:-270:0.15);
}
\end{scope}

\begin{scope}[xshift=3.5cm, yshift=-1.5cm]
\draw[->, thick] (0,0) -- (3,0);
\node[below] at (2.5,-0.2) {$\rho$};
\draw[->, thick] (0,0) -- (0,3) node[left] {$\mathcal{A}$};

\foreach \y in {0,1,2} {
    \draw (0.05,\y) -- (-0.05,\y);
    \node[left] at (-0.15,\y) {$\y$};
}

\draw[blue, dashed, thick] (0.75,0) -- (0.75,2.0);
\node[blue, below] at (0.75,-0.3) {$\rho = \frac{1}{4}$};

\draw[red, very thick, line cap=round] (0.75,0.75) -- (0.75,2.0);

\fill[red] (0.75,0.75) circle (2pt) node[right] {$\mathcal{A} = \frac{1}{4}$};
\fill[red] (0.75,2) circle (2pt) node[right] {$\mathcal{A} = 2$};

\end{scope}

\end{tikzpicture}


\

In particular, there is no example in which that the x-axis is a single irrational point and the y-axis is an interval, as shown in \cite{Deng21}.

\section{Further problems}

Denote $\mathcal{M}_{f}$ as the set of f-invariant measure on a surface $S$, $\mathcal{M}_{f}$ is a convex set. Consider 
$$\mathcal{A}: \mathcal{M}_{f}\to\R, \ \ \ \rho: \mathcal{M}_{f}\to\R$$
as two linear functionals, then the image of these two functionals is also convex. Denote the Lebesgue measure of the volume as $m$, if we change the $\beta$ in the definition of $\A$,  the graph of these two functionals will rotate and translate by a constant,  based on our previous analysis about the action and rotation number.


In \cite{Deng21}, Xia and Deng proved that action difference, similar to rotation number difference, can be seen as a kind of "twist condition" that implies the existence of infinitely many periodic orbits, but the existence of periodic orbit with action value between the difference action values is unidentified. The results on disk and annulus show the existence of invariant measure supported on periodic orbits such that its action between $\mathcal{A}(m)$ and the action of invariant measure supported on the boundary. It's natural to surch for a more general results, perhaps without necessarily relying on information from high-dimensional contact geometry. For instance, the approach taken by Le Calvez in \cite{Le23} may provide a way forward.


\section*{Acknowledgments}

The author wishes to thank Zhihong Xia for proposing the subject of this research, and thank Wentian Kuang, Jian Wan and Jingzhi Yan for their helpful discussions, and especially David Bechara Senior for bringing example 4.3 to the author's attention.

\end{document}